\documentclass[11pt]{amsart}

\usepackage{amsmath,amssymb,amsthm,amsrefs,a4wide}
\usepackage{latexsym}
\usepackage{indentfirst}
\usepackage{graphicx}
\usepackage{placeins}
\usepackage{booktabs}
\usepackage{algorithm}
\usepackage{algorithmic}
\usepackage{multirow}

\theoremstyle{plain}
\newtheorem{theorem}{Theorem}

\theoremstyle{definition}

\theoremstyle{remark}
\newtheorem{remark}{Remark}

\DeclareMathOperator{\polylog}{polylog}

\newcommand{\eps}{\epsilon}

\newcommand{\bbm}{\begin{bmatrix}}
\newcommand{\ebm}{\end{bmatrix}}

\newcommand{\pq}{\frac{p}{q}}
\newcommand{\wpwq}{\frac{1-p}{1-q}}
\newcommand{\hot}{\text{h.o.t.}}

\begin{document}

\title[Low-rank Structure]{Hierarchical Low-rank Structure of Parameterized Distributions}

\author[]{Jun Qin}
\address[Jun Qin]{Target Corporation, Sunnyvale, CA, 94086}
\email{jun.qin@target.com}

\author[]{Lexing Ying}
\address[Lexing Ying]{Department of Mathematics and ICME, Stanford University, Stanford, CA 94305}
\email{lexing@stanford.edu}

\thanks{The work of L.Y. is partially supported by the U.S. Department of Energy, Office of Science,
  Office of Advanced Scientific Computing Research, Scientific Discovery through Advanced Computing
  (SciDAC) program and also by the National Science Foundation under award DMS-1818449.}


\begin{abstract}
  This note shows that the matrix forms of several one-parameter distribution families satisfy a
  hierarchical low-rank structure. Such families of distributions include binomial, Poisson, and
  $\chi^2$ distributions. The proof is based on a uniform relative bound of a related divergence
  function. Numerical results are provided to confirm the theoretical findings.
\end{abstract}

\maketitle

\section{Introduction}\label{sec:intro}

This note is concerned with the matrix or operator form $f(x,\lambda)$ of a one-parameter
distribution family indexed by the parameter $\lambda$. Such objects have long been considered in
Bayesian statistics \cite{gelman2013bayesian,wasserman2013all,ghosal2017fundamentals}. More
recently, these matrices have played an important role in estimating distributions of distributions
(also called fingerprints) \cite{valiant2013estimating,tian2017learning} and computing functionals
of unknown distributions from samples
\cite{jiao2015minimax,paninski2003estimation,wu2016minimax}. When solving these problems, the
computation often requires solving linear systems and optimizations problems associated with these
matrices and operators.

In this note, we prove that, for several one-parameter family of distributions, including binomial,
Poisson, and $\chi^2$ distributions, $f(x,\lambda)$ exhibits a hierarchical low-rank
structure. Roughly speaking, when viewed as a two-dimensional array, the off-diagonal blocks of
$f(x,\lambda)$ are numerically low-rank, i.e., for a fixed accuracy $\eps$, the numerical rank is
bounded by a poly-logarithmic function of $1/\eps$. Such a structure ensures optimal complexity
while approximating these matrices or performing basic linear algebra operations such as
matrix-vector multiplications. In order to demonstrate the existence of such low-rank
approximations, we first prove a new relative bound for a related divergence function, which might
be of independent interest.

Similar hierarchical low-rank properties have been demonstrated for integral kernels
\cite{rokhlin1985rapid,greengard1988rapid,greengard1991fast,bebendorf2003existence,hackbusch2015hierarchical}
related to partial differential equations. For those kernels, the difficulty comes from the
singularity along the diagonal. For the problems considered in this note, the location of the
singularity is often near the boundary of the matrix/operator and thus the proof technique is quite
different.

The rest of the note is organized as follows. Section \ref{sec:rela} proves a relative bound of a
related divergence function. Section \ref{sec:lowrank} discusses the hierarchical low-rank structure
of the negative exponentials of the divergence functions. Finally in Section \ref{sec:dist} extends
this result to parameterized distributions, including the binomial, Poisson, and $\chi^2$ squared
distributions.

\section{A relative bound for a divergence function} \label{sec:rela}


Consider the divergence function
\[
E(p||q) \equiv p \ln(p/q) - (p-q)
\]
for $0\le p,q < \infty$, which is convex and positive away from $p=q$. Let us first focus on the
square $(p,q)\in (1,2)\times(0,1)$.

\begin{theorem}\label{thm:Elower}
  For any $M>0$, define $p_M$ and $q_M$ as follows:
  \begin{itemize}
  \item $q_M<1$ is the value such that $E(1||q_M)=\ln 1/q_M - (1-q_M)=M$.
  \item $p_M=\min(2,p')$ where $p'>1$ is the number such that $E(p'||1)=p' \ln p' - (p'-1) = M$.
  \end{itemize}
  There exists a uniform constant $C>0$ such that for any $M>0$
  \[
  \frac{E(p_M||q_M)}{M}<C.
  \]  
\end{theorem}

\begin{figure}[h!]
  \centering
  \includegraphics[scale=0.42]{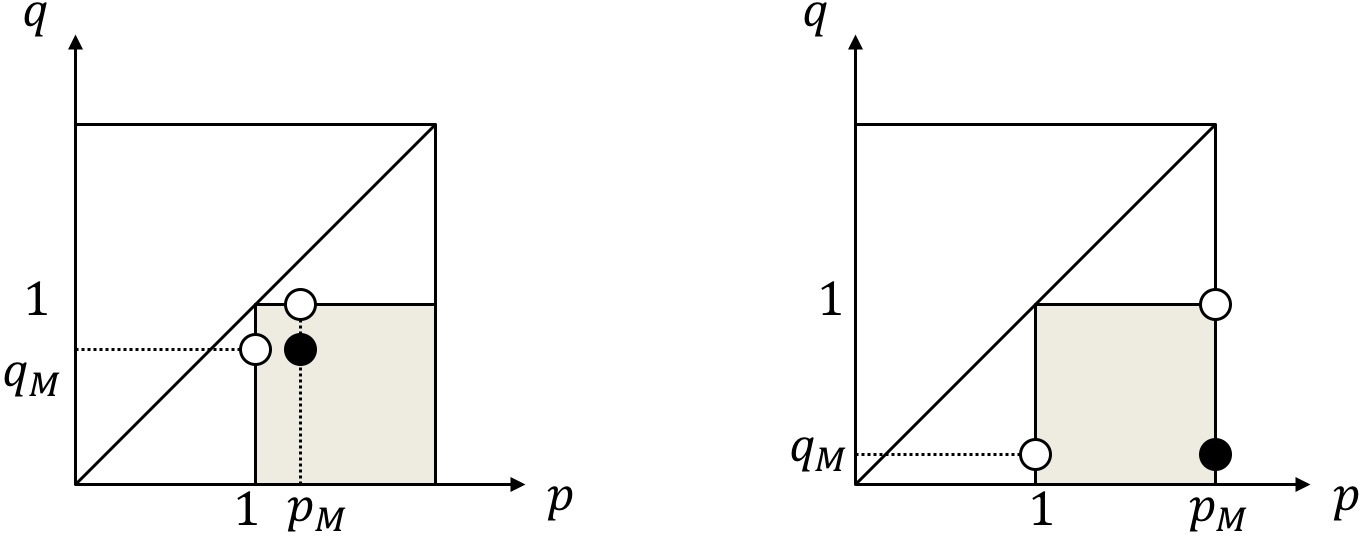}
  \caption{Locations of $p_M$ and $q_M$ for the case $(p,q)\in(1,2)\times(0,1)$.  Left: small
    $M$. Right: large $M$. }
  \label{fig:Elower}
\end{figure}

\begin{proof}
  Since the ratio $E(p_M||q_M)/M$ depends continuously on $M$, in order to show it is bounded by a
  uniform constant, it is sufficient to show that the ratio $E(p_M||q_M)/M$ has a finite limit as
  $M$ goes to zero and to infinity.
  
  When $M$ goes to zero, the Taylor approximation of $E(p||q)$ near $p=1$ and $q=1$ is valid. The
  first order derivatives of $E(p||q)$ are
  \[
  E_p = \ln p-\ln q, \quad  E_q = -p/q + 1,
  \]
  while the second order derivatives are
  \[
  E_{pp} = 1/p, \quad  E_{pq} = -1/q, \quad E_{qq}=p/q^2.
  \]
  At the point $(p,q)=(1,1)$,
  \[
  E_p|_{(1,1)}=E_q|_{(1,1)}=0,\quad
  E_{pq}|_{(1,1)}=E_{qq}|_{(1,1)}=1,\quad
  E_{pq}|_{(1,1)}=-1.
  \]
  Applying the definition of $p_M$ and $q_M$ shows that when $M$ goes to zero
  \[
  p_M = 1+\sqrt{2M} + \hot \quad
  q_M = 1-\sqrt{2M} + \hot
  \]
  where $\hot$ stands for higher order terms (see Figure \ref{fig:Elower} (left)). Plugging them
  back to $E(p_M ||q_M)$ and using the second order Taylor approximation shows
  \[
  E(p_M || q_M) = 4M + \hot
  \]
  Therefore, when $M$ goes to zero, the ratio $E(p_M||q_M)/M$ goes to 4.

  When $M$ goes to infinity, $p_M$ goes to 2. From the definition, $q_M$ satisfies
  \[
  \ln(1/q_M) - (1-q_M) = M.
  \]
  Therefore, $q_M = e^{-(M+1)}(1+\hot)$ (see Figure \ref{fig:Elower} (right)).  Plugging them back
  to $E(p_M ||q_M)$ shows that
  \[
  E(p_M||q_M) = p_M \ln p_M/q_M - (p_M-q_M) =
  2\ln2  + 2(M+1) - 2 + \hot
  \]
  When $M$ goes to infinity, the ratio $E(p_M||q_M)/M$ goes to 2.
  
  Putting these two cases together proves the statement.
\end{proof}

Next, consider the square $(p,q)\in(0,1)\times(1,2)$.

\begin{theorem}\label{thm:Eupper}
  For any $M>0$, now define $p_M$ and $q_M$ as follows:
  \begin{itemize}
  \item $q_M=\min(2,q')$ where $q'>1$ satisfies $E(1||q')=\ln(1/q') - (1-q')=M$.
  \item $p_M$ is the minimum $p'\ge 0$ with $E(p'||1)=p' \ln p' - (p'-1) \le M$.
  \end{itemize}
  There exists a uniform constant $C>0$ such that for any $M>0$
  \[
  \frac{E(p_M||q_M)}{M}<C.
  \]  
\end{theorem}

\begin{figure}[h!]
  \centering
  \includegraphics[scale=0.42]{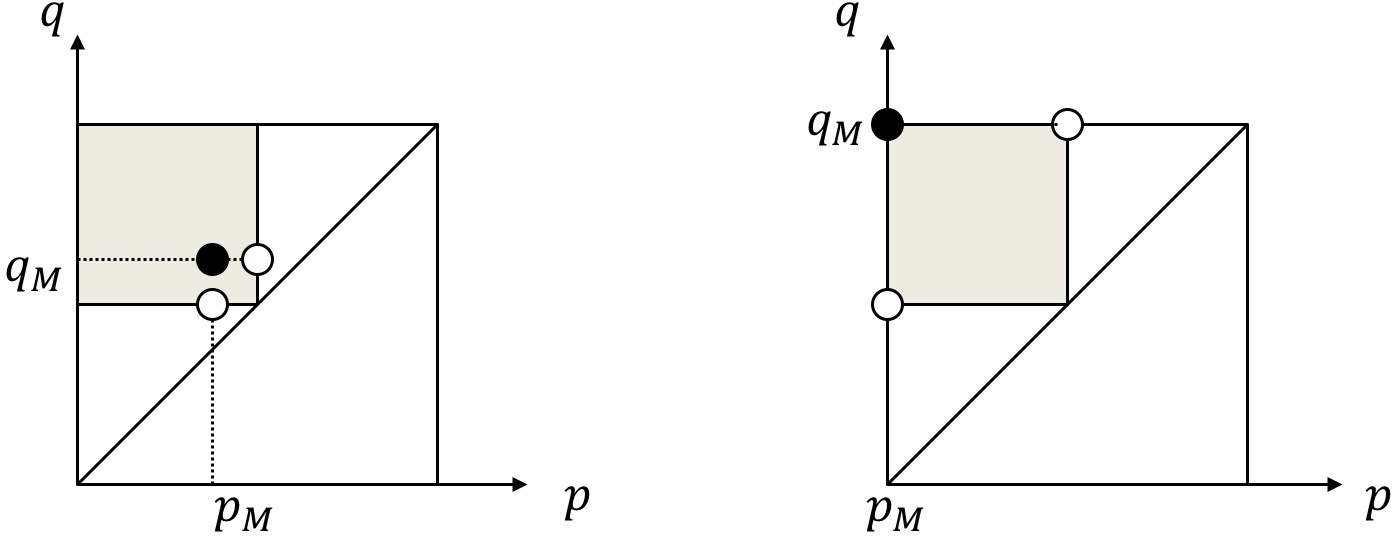}
  \caption{Locations of $p_M$ and $q_M$ for the case $(p,q)\in(0,1)\times(1,2)$.  Left: small
    $M$. Right: large $M$. }
  \label{fig:Eupper}
\end{figure}

\begin{proof}
  Following the proof of the previous theorem, it is sufficient to show that the ratio has a limit
  as $M$ goes to zero and infinity.

  When $M$ goes to zero, one can again use the second order Taylor expansion. Applying the
  definition of $p_M$ and $q_M$, for sufficiently small $M$,
  \[
  p_M = 1-\sqrt{2M} + \hot\quad
  q_M = 1+\sqrt{2M} + \hot
  \]
  (see Figure \ref{fig:Eupper} (left)). Plugging them back to $E(p_M ||q_M)$ and using again the
  Taylor approximation shows
  \[
  E(p_M || q_M) = 4M + \hot
  \]
  Therefore, when $M$ goes to zero, the ratio $E(p_M||q_M)/M$ goes to 4.

  When $M$ goes to infinity, $p_M$ goes to 0 and $q_M$ goes to 2 (see Figure \ref{fig:Eupper}
  (right)). Plugging them back to $E(p_M ||q_M)$ shows that
  \[
  E(p_M||q_M) = 2 + \hot
  \]
  Therefore, when $M$ goes to infinity, the ratio $E(p_M||q_M)/M$ goes to 0.
  
  Putting these two cases together proves the statement.
\end{proof}

\begin{remark}
  Theorems \ref{thm:Elower} and \ref{thm:Eupper} also hold for the dual divergence of $E$ defined as
  \[
  E^*(p||q) = q \ln(q/p) - (q-p)
  \]
  for $0<p,q<\infty$ by simply switching the roles of $p$ and $q$.
\end{remark}

\section{Hierarchical low-rank structure of negative exponential of divergence} \label{sec:lowrank}

\subsection{Divergence $E(p||q)$}\label{sec:E}

Consider now the negative exponential of the divergence $E(p||q)$
\begin{equation}\label{eq:enE}
  \exp(-n E(p||q) = \exp(-n (p\ln(p/q) - (p-q)))
\end{equation}
for $0\le p,q<\infty$ and any $n>0$. 

We consider a hierarchical decomposition that partitions the domain $(p,q)\in (0,\infty)^2$ into
non-overlapping squares in a multiscale way. For each level $\ell$ indexed by integers, introduce the
blocks $B_{\ell,k}$ defined as follows for $k=0,1,\ldots$,
\[
B_{\ell,k} =
\begin{cases}
  [k/2^\ell,(k+1)/2^\ell] \times [(k+1)/2^\ell,(k+2)/2^\ell],& \text{for $k$ even,}\\
  [k/2^\ell,(k+1)/2^\ell] \times [(k-1)/2^\ell,k/2^\ell],    & \text{for $k$ odd}.
\end{cases}
\]
An illustration of this partitioning is shown in Figure \ref{fig:ED} (left).

\begin{figure}[h!]
  \centering
  \includegraphics[scale=0.42]{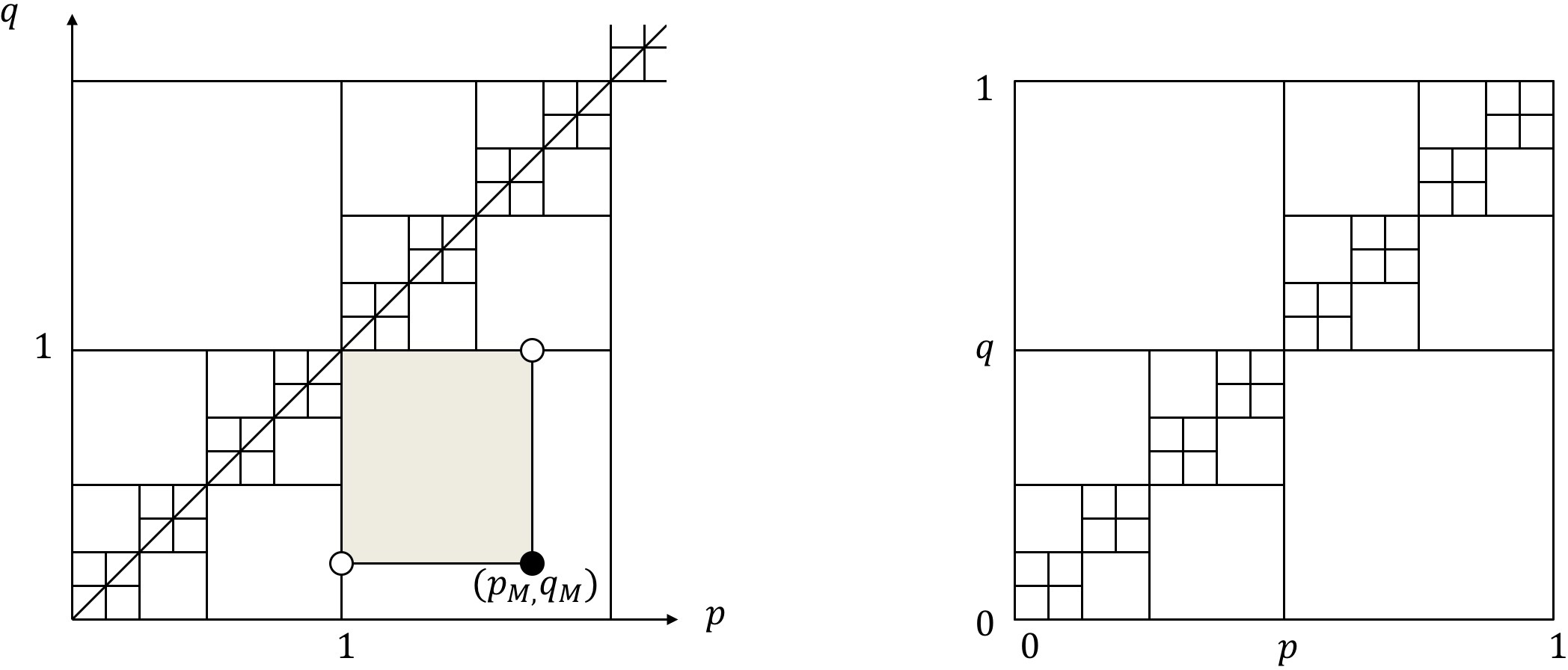}
  \caption{Left: Hierarchical decomposition for $\exp(-n E(p||q))$ for $(p,q)\in (0,\infty)^2$.
    Right: Hierarchical decomposition for $\exp(-n D(p||q))$ for $(p,q)\in (0,1)^2$.
  }
  \label{fig:ED}
\end{figure}

The main goal of this section is to prove the following theorem concerning the numerical rank of
\eqref{eq:enE} restricted to each $B_{\ell,k}$.

\begin{theorem}\label{thm:E}
  For any $\eps>0$, there exists a constant $T_\eps = O(\polylog(1/\eps))$ such that for any $n>0,
  \ell, k$ the restriction of $\exp(-n E(p||q))$ to $B_{\ell,k}$ has an $O(\eps)$-accurate
  $T_\eps$-term separated approximation. More precisely, there exists functions $\{\alpha_i(p)\}$ and
  $\{\beta_i(q)\}$ for $1\le i \le T_\eps$ such that in $B_{\ell,k}$
  \[
  \exp(-n E(p||q)) = \sum_{i=1}^{T_\eps} \alpha_i(p)\beta_i(q) + O(\eps).
  \]
\end{theorem}

\begin{proof}
  
  Consider first the blocks $B_{\ell,k}$ with $k$ odd. These blocks are below the diagonal
  $p=q$. The top left corner of $B_{\ell,k}$ is $(c,c)$ with $c=k/2^\ell$.  Let us make two key
  observations.
  \begin{itemize}
  \item It is sufficient to prove the theorem for the restriction of $\exp(-n E(p||q))$ to
    $(c,2c)\times(0,c)$ as the later contains $B_{\ell,k}$.
  \item The second observation is that, as the statement is uniform in $n$, it is sufficient to
    scale the box $(c,2c)\times(0,c)$ to $(1,2)\times(0,1)$ by scaling the value of $n$ accordingly.
  \end{itemize}
  Based on these two observations, it is sufficient to consider the box $(1,2)\times(0,1)$ for any
  $\eps>0$ and any $n>0$.

  For fixed $\eps>0$ and $n>0$, define $M=\frac{1}{n}\ln\frac{1}{\eps}$. Applying Theorem
  \ref{thm:Elower} along with the definition of $p_M$ and $q_M$ gives
  \[
  E(p_M||q_M) \le C M = C\frac{1}{n}\ln\frac{1}{\eps}
  \]
  and by monotonicity
  \[
  E(p||q) \le CM = C\frac{1}{n}\ln\frac{1}{\eps}, \quad \forall (p,q) \in [1,p_M]\times[q_M,1].
  \]
  In order to construct a separated approximation of $\exp(-n (p\ln(p/q) - (p-q)))$, we resort to
  the polynomial expansion for $(p,q) \in [1,p_M]\times[q_M,1]$.
  
  In order for this, consider the function $\exp(-x)$ in $x\in[0,L]$ for some $L>0$. Using the
  Lagrange interpolation at the Chebyshev grids in $[0,L]$ and the uniform bound of the derivatives
  of $\exp(-x)$ (see for example Theorem 8.7 of \cite{suli2003introduction}), we know that there
  exists a degree $d=O(\ln L + \ln(1/\eps))$ polynomial $h_d(x)$ such that
  \[
  \exp(-x) - h_d(x) = O(\eps).
  \]
  Plugging $x=nE(p,q)$ for $(p,q) \in [1,p_M]\times[q_M,1]$ with the bound $L=n\cdot
  C\frac{1}{n}\ln\frac{1}{\eps} = C\ln(1/\eps)$, one arrives at
  \[
  \exp(-nE(p||q)) - h_d(nE(p||q)) = O(\eps),
  \]
  with $d=O(\ln(1/\eps))$. As $E(p||q) = p \ln p - p\ln q -p + q$, by expanding the polynomial
  $h_d(\cdot)$, we obtain a $O(\polylog(1/\eps))$-term separated approximation to $\exp(-nE(p||q))$
  for $(p,q)\in[1,p_M]\times[q_M,1]$. The individual terms define the functions $\{\alpha_i(p)\}$
  for $p\in[1,p_M]$ and $\{\beta_i(q)\}$ for $q\in[q_M,1]$, respectively.

  For any point $(p,q)\in (1,2)\times(0,1)$ but outside $[1,p_M]\times[q_M,1]$, as
  $\exp(-nE(p||q))\le \eps$, one can simply approximate it by zero without introducing an error
  larger than $\eps$. In terms of the functions $\alpha_i(p)$ and $\beta_i(q)$, we simply define
  $\alpha_i(p)$ to be zero in $[p_M,2]$ and $\beta_i(q)$ to be zero in $q\in[0,q_M]$, respectively.
  
  Next, we consider the blocks $B_{\ell,k}$ with $k$ even. These are the blocks above the diagonal
  $p=q$. The above argument goes through except that Theorem \ref{thm:Eupper} is invoked.
  
\end{proof}

\begin{remark}
  The same theorem is true for 
  \[
  \exp(-n E^*(p||q)) \equiv \exp(-n (q\ln(q/p) - (q-p))),
  \]
  for $0<p,q<\infty$ by switching the roles of $p$ and $q$.
\end{remark}

\begin{remark}
  The same theorem is true for
  \[
  \exp(-n E(1-p||1-q))
  \]
  for $-\infty<p,q<1$ with a similar hierarchical partitioning of the domain $-\infty<p,q<1$.
\end{remark}

\subsection{Kullback-Leibler divergence}\label{sec:KL}

The Kullback-Leibler (KL) divergence of two Bernoulli distributions with parameters $p,q\in[0,1]$ is
defined as
\[
D(p||q) \equiv p\ln (p/q) + (1-p)\ln ((1-p)/(1-q)).
\]
This section proves the hierarchical low-rank property for
\[
\exp(-n D(p||q) = \exp(-n (p\ln(p/q) + (1-p)\ln ((1-p)/(1-q)) ))
\]
with $0<p,q<1$. For the domain $(p,q)\in [0,1]\times[0,1]$, the hierarchical decomposition needs to
be restricted to
\[
\ell\ge 1, \quad
k = 0,1,\ldots, 2^\ell-1.
\]
An illustration of this partitioning is shown in Figure \ref{fig:ED} (right).

\begin{theorem}\label{thm:KL}
  For any $\eps>0$, there exists a constant $S_\eps = O(\polylog(1/\eps))$ such that for any $n > 0,
  \ell\ge 1, k = 0,1,\ldots, 2^\ell-1$, the restriction of $\exp(-n D(p||q))$ to $B_{\ell,k}$ has an
  $O(\eps)$-accurate $S_\eps$-term separated approximation. More precisely, there exists functions
  $\{\alpha_i(p)\}$ and $\{\beta_i(q)\}$ for $1\le i \le S_\eps$ such that in $B_{\ell,k}$
  \[
  \exp(-n D(p||q)) = \sum_{i=1}^{S_\eps} \alpha_i(p)\beta_i(q) + O(\eps).
  \]
\end{theorem}

\begin{proof}
  The proof is based on a simple observation: $D(p||q) = E(p||q) + E(1-p||1-q)$, which implies \[
  \exp(-nD(p||q)) = \exp(-nE(p||q))\exp(-nE(1-p||1-q)).
  \]
  From Theorem \ref{thm:E} and the remarks right after, the following two estimates hold
  for each $B_{\ell,k}$.
  \begin{align*}
    \exp(-n E(p||q)) &= \sum_{i=1}^{T_\eps} \alpha_i(p)\beta_i(q) + O(\eps),\\ \exp(-n E(1-p||1-q))
    &= \sum_{j=1}^{T_\eps} \alpha'_j(p)\beta'_j(q) + O(\eps),
  \end{align*}
  Taking the product of these two expansions and using the fact that each expansion is bounded by
  $1+O(\eps)$ results in
  \[
  \exp(-nD(p||q)) =  \sum_{i,j=1}^{T_\eps} (\alpha_i(p)\alpha'_j(p)) (\beta_i(q) \beta'_j(q)) + O(\eps).
  \]
  Noticing that $T_\eps^2$ is still of order $O(\polylog(1/\eps))$, setting $S_\eps = T_\eps^2$
  completes the proof.
\end{proof}

\section{Parameterized distributions} \label{sec:dist}

In this section, we apply the theorems in Section \ref{sec:lowrank} to demonstrate the hierarchical
low-rank property for three commonly-encountered distribution families.

\subsection{Binomial distribution}

The binomial distribution with parameter $q\in[0,1]$ and $n$ trials is
\[
f(k,q) = {n \choose k} q^k (1-q)^{n-k},
\]
for $k\in\{0,\ldots,n\}$. By introducing $p=k/n$, we can rewrite the binomial distribution in the
form
\[
f(p,q) = {n\choose np} q^{np} (1-q)^{n-np}
\]
with $p=0,\frac{1}{n},\ldots,1$. Applying the Stirling formula to the factorials results in
\[
f(p,q) = c_{n,p} \frac{q^{np} (1-q)^{n-np}}{p^{np}(1-p)^{n-np}}
=c_{n,p} \exp\left(-n\left(p\ln\pq + (1-p)\ln\wpwq\right)\right),
\]
where $c_{n,p} \approx \frac{1}{\sqrt{2\pi n}} \frac{1}{\sqrt{p(1-p)}}$ except at $p=0$ and $p=1$.

Applying Theorem \ref{thm:KL} to this case shows that the binomial distribution $f(p,q)$ for
$p=0,1/n,\ldots,1$ and $q\in[0,1]$ has the hierarchical low-rank property. Here the two points $p=0$
and $1$ can be treated separately without affecting the rank estimates. Figure \ref{fig:binores}
plots the numerical rank of different blocks for a specific choice of $n$ and $\eps$ (left) and its
dependence on $\eps$ (right). Note that the rank is bounded by $10$ even for $\eps=10^{-9}$ and the
dependence of the rank on $\ln(1/\eps)$ is linear.

\begin{figure}[h!]
  \centering
  \includegraphics[width=0.45\textwidth]{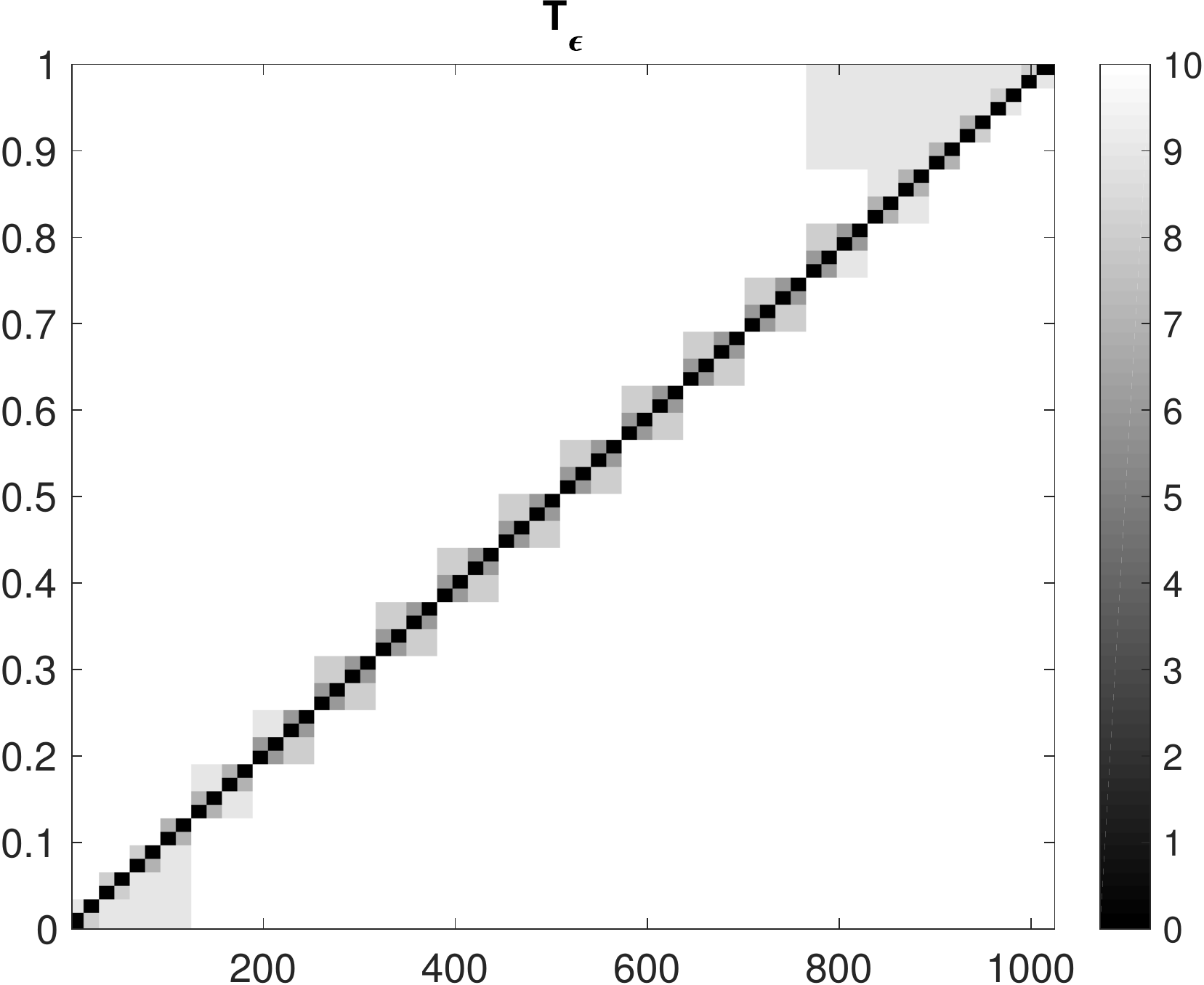}
  \includegraphics[width=0.45\textwidth]{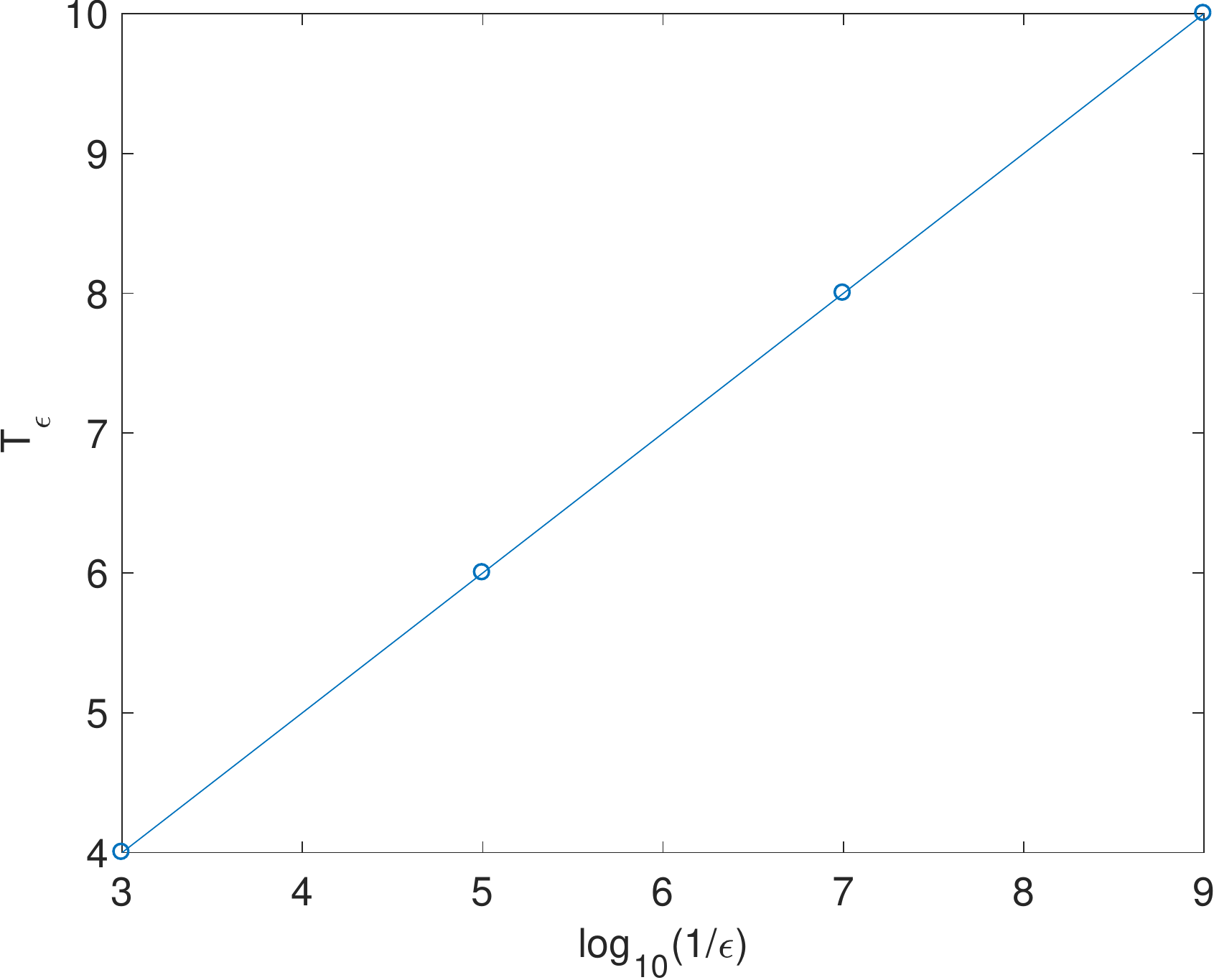}
  \caption{Binomial distribution. Left: the numerical rank $T_\eps$ of different blocks with
    $n=2^{10}$ and $\eps=10^{-9}$.  Right: the maximum of the numerical ranks $T_\eps$ as a function
    of $\eps$ with $n=2^{10}$.  }
  \label{fig:binores}
\end{figure}

\subsection{Poisson distribution}

The Poisson distribution with parameter $\lambda>0$ is
\[
f(k,\lambda) = e^{-\lambda} \frac{\lambda^k}{k!}
\]
for $k\in\{0,1,\ldots\}$. Applying the Stirling formula to $k!$ gives for $k>0$
\[
f(k,\lambda) \approx \frac{1}{\sqrt{2\pi k}} \exp(-(k\log(k/\lambda) - (k-\lambda))).
\]
By identifying $p=k$ and $q=\lambda$, this is the negative exponential of the divergence $E(p||q)$
with $n=1$ in Section \ref{sec:E}, modulus the term $\frac{1}{\sqrt{2\pi k}}$.

Applying Theorem \ref{thm:E} shows that the Poisson distribution $f(k,\lambda)$ for $k=0,1,\ldots$
and $\lambda>0$ exhibits the hierarchical low-rank property. Figure \ref{fig:poissonres} shows the
numerical rank of different blocks for a specific choice of $\eps$ (left) and its dependence on
$\eps$ (right). Note that the rank is bounded by $10$ even for $\eps=10^{-9}$ and there is strong
evidence that the dependence of the rank on $\ln(1/\eps)$ is linear.

\begin{figure}[h!]
  \centering
  \includegraphics[width=0.45\textwidth]{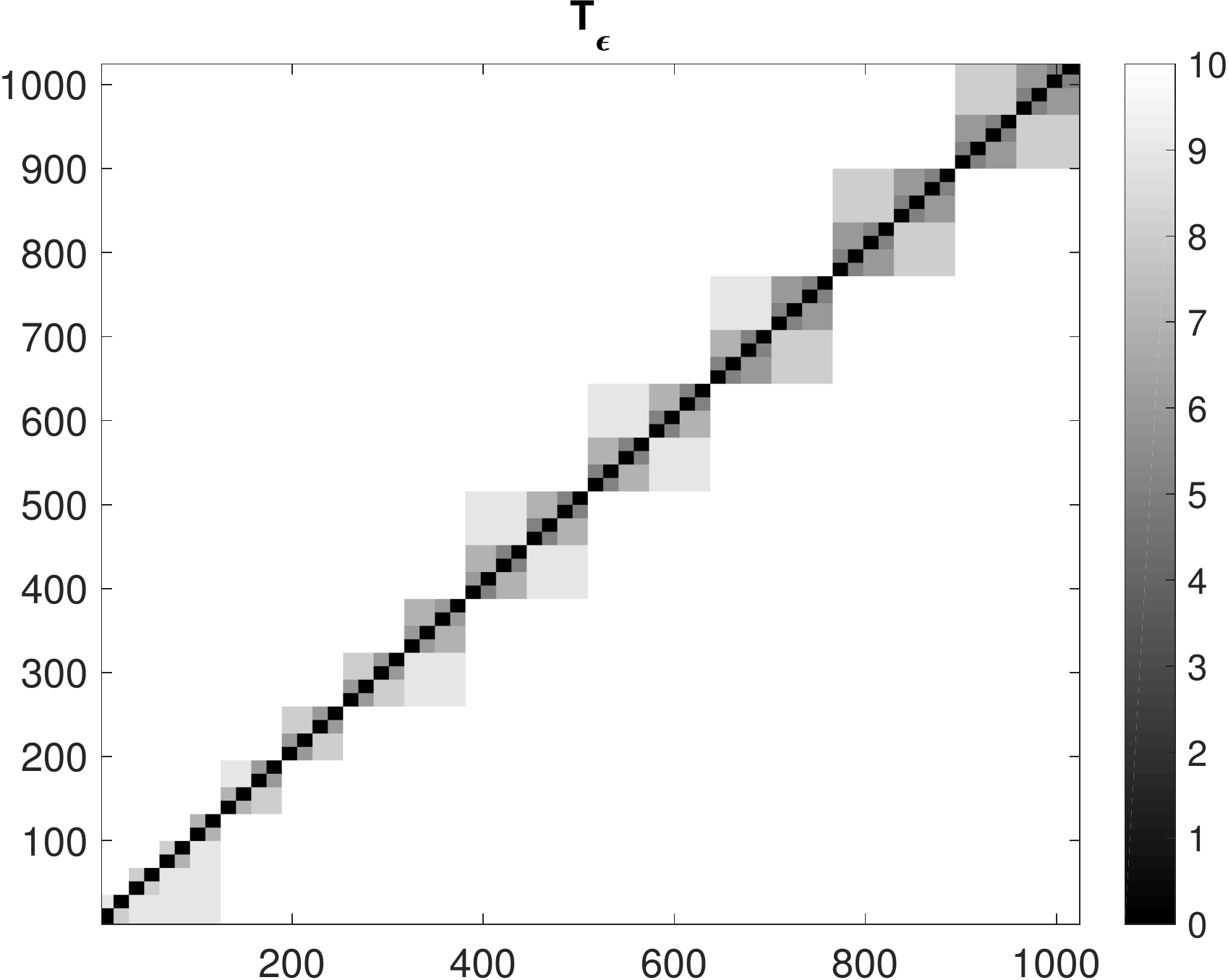}
  \includegraphics[width=0.45\textwidth]{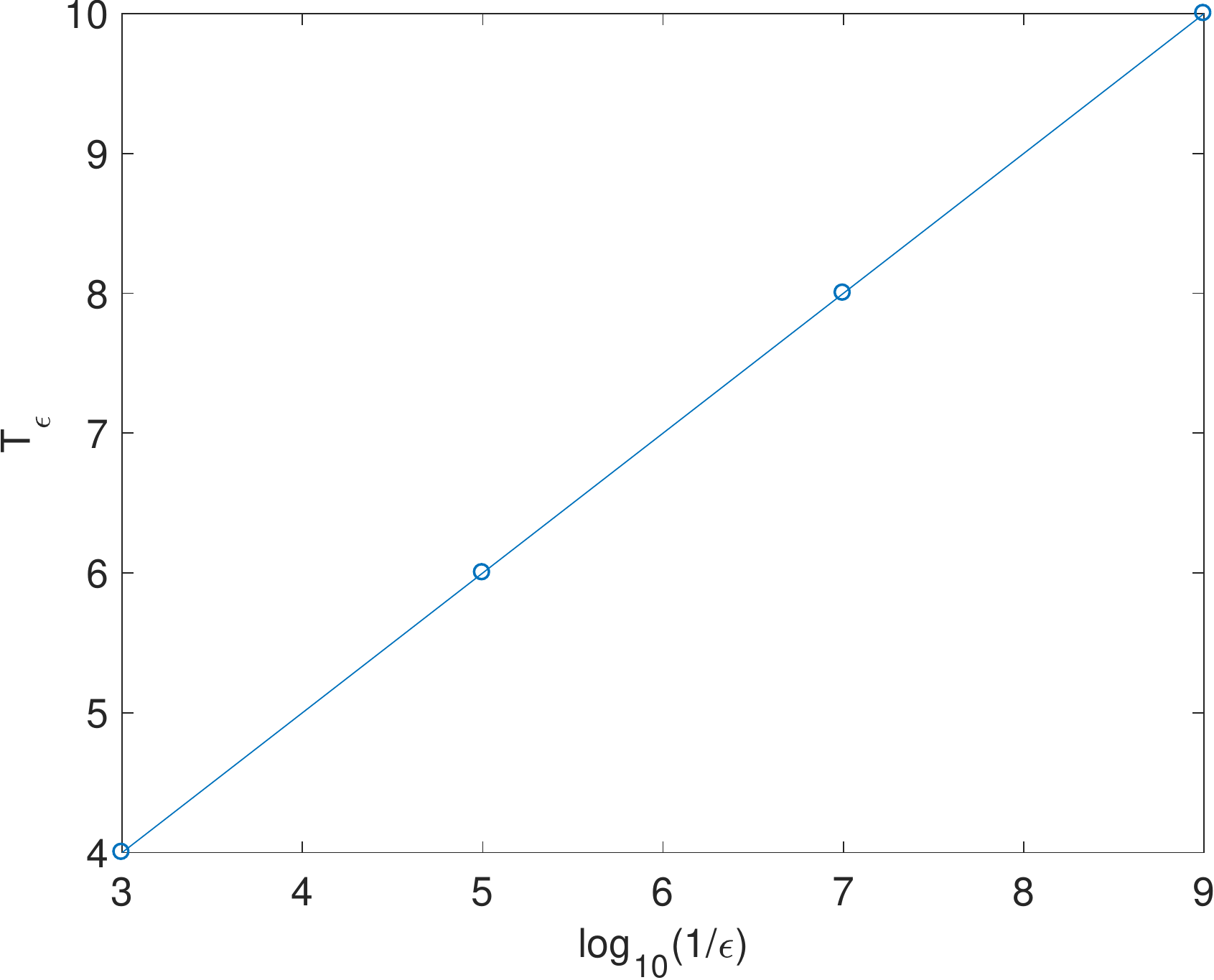}
  \caption{Poisson distribution. Left: numerical rank $T_\eps$ of different blocks with
    $k,\lambda\le 2^{10}$ and $\eps=10^{-9}$.  Right: the maximum of the numerical ranks $T_\eps$ as
    a function of $\eps$.}
  \label{fig:poissonres}
\end{figure}

\subsection{$\chi^2$ distribution}

The $\chi^2$ distribution, parameterized by integer $k\ge 1$ is 
\[
f(x,k) = \frac{1}{2^{k/2}\Gamma(k/2)} x^{k/2-1} e^{-x/2}.
\]
for $x>0$. Applying again the Stirling formula shows that
\[
f(x,k) \approx \frac{1}{2\sqrt{2\pi\left(\frac{k}{2}-1\right)}}
\exp\left(-\left(\frac{x}{2}-\left(\frac{k}{2}-1\right) +
\left(\frac{k}{2}-1\right)\ln\left(\frac{ k/2-1}{x/2}\right) \right)\right).
\]
By identifying $x/2=p$ and $k/2-1=q$, this is
\[
\exp(-(q\ln(q/p)-(q-p)))
\]
modulus the factor $\frac{1}{2\sqrt{2\pi(k/2-1)}}$.

Applying the remark after Theorem \ref{thm:E} shows that the $\chi^2$ distribution exhibits the
hierarchical low-rank property. Figure \ref{fig:chi2res} plots the numerical rank of different
blocks for a specific choice of $\eps$ (left) and its dependence on $\eps$. Again, the rank remains
small even for $\eps=10^{-9}$ and the dependence of the rank on $\ln(1/\eps)$ is linear.

\begin{figure}[h!]
  \centering
  \includegraphics[width=0.45\textwidth]{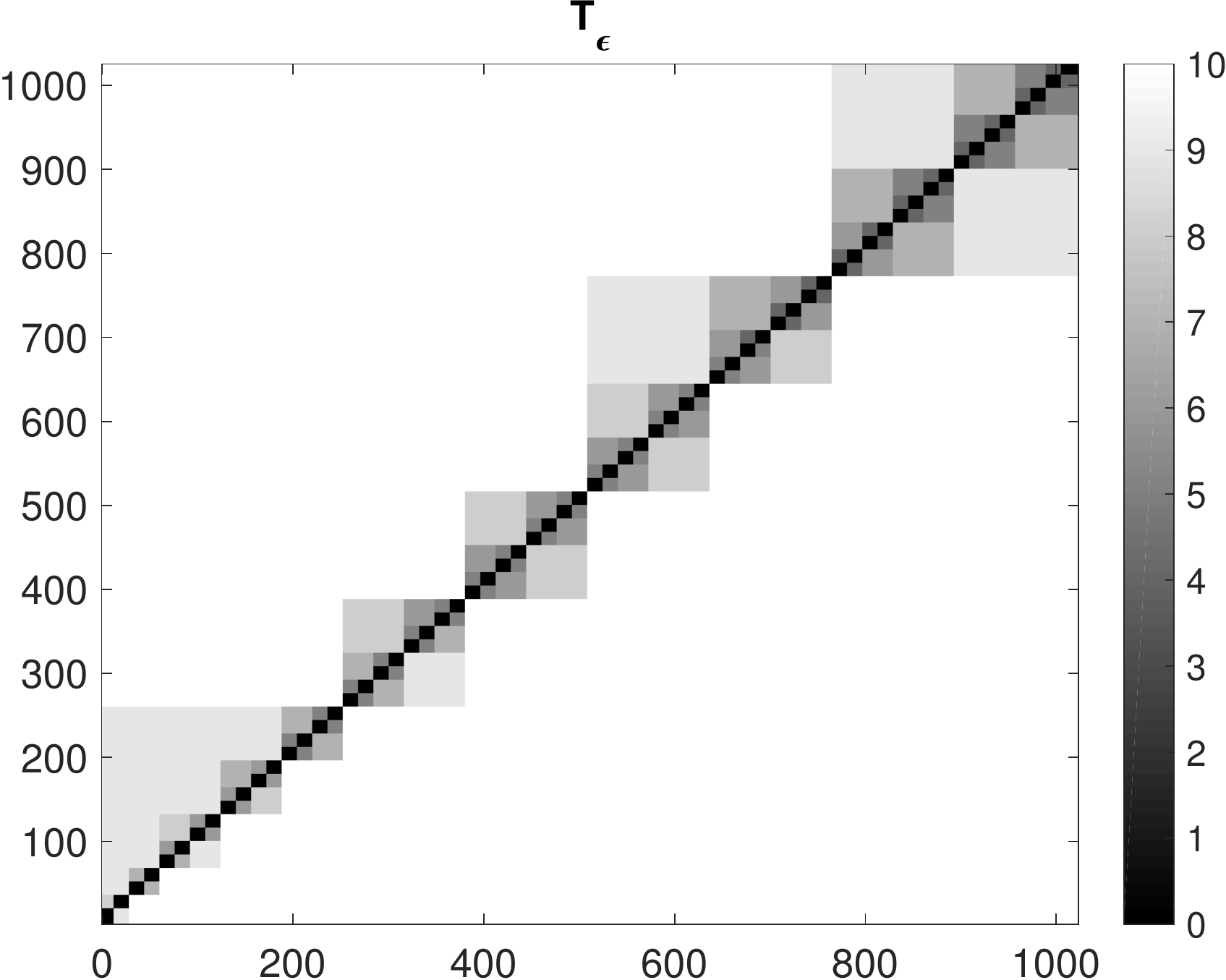}
  \includegraphics[width=0.45\textwidth]{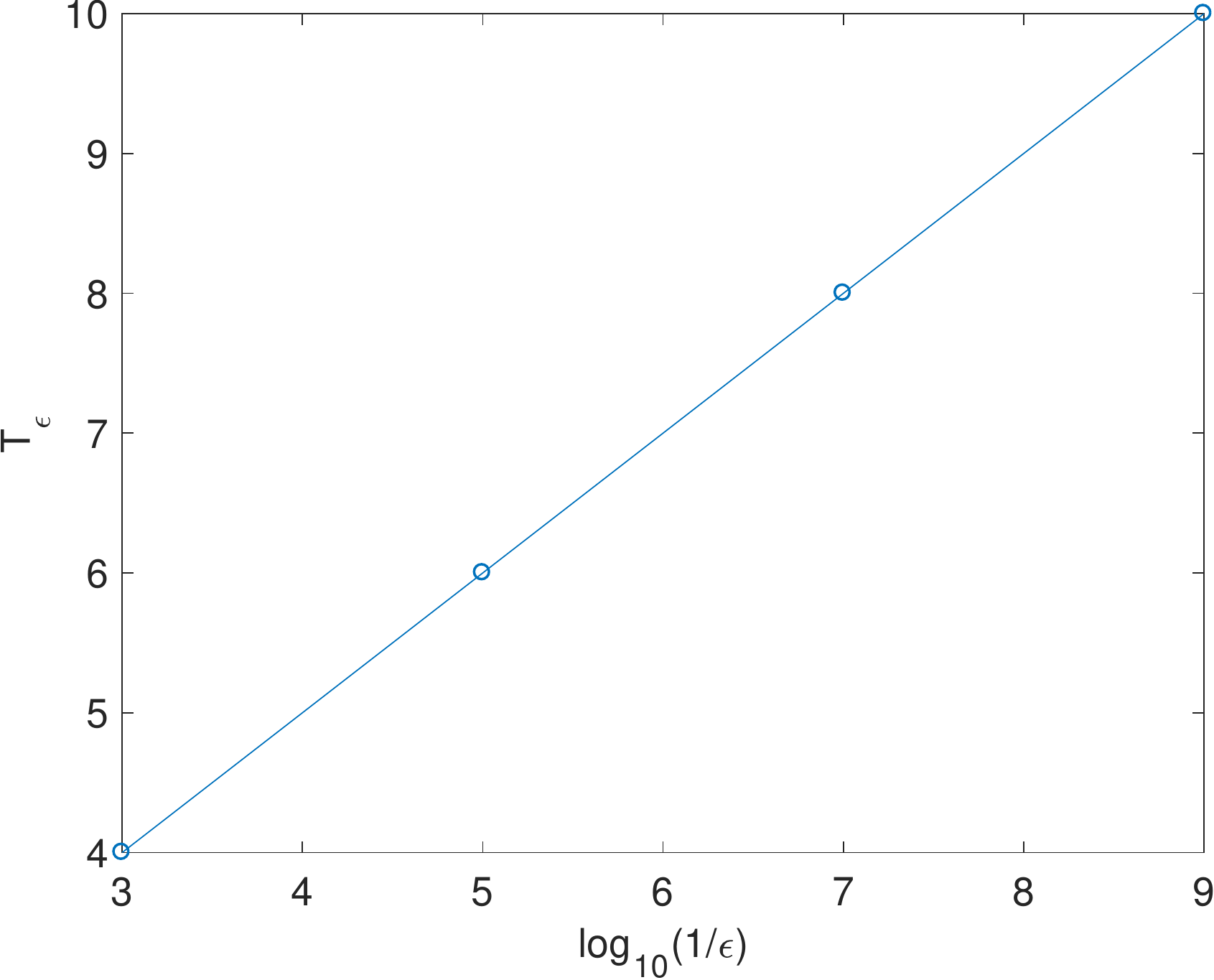}
  \caption{$\chi^2$ distribution. Left: numerical rank $T_\eps$ of different blocks with
    $x,k\lesssim 2^{10}$ and $\eps=10^{-9}$.  Right: the maximum of the numerical ranks $T_\eps$ as
    a function of $\eps$.  }
  \label{fig:chi2res}
\end{figure}

\section{Discussions}

The hierarchical low-rank property has significant numerical implications for these distribution
families. Naive approaches for representing the matrix form of these distributions would require
$O(n^2)$ numbers. Even by thresholding small entries, it would still need at least $O(n^{3/2})$
storage space for most of these distributions. The hierarchical low-rank property proved here allows
for storing the matrix with no more than $O(n\log n \polylog(1/\eps))$ entries. By combining the
low-rank property with thresholding, this can potentially be brought down to $O(n\polylog(1/\eps))$.

The theorems proved here show an $O(\polylog(1/\eps))$ upper bound for the numerical ranks. However,
the numerical results suggest that the actual dependence on $\log(1/\eps)$ seems to be linear. An
immediate direction for future work is to obtain sharper bounds for the rank growth.

\bibliographystyle{abbrv}

\bibliography{ref}

\end{document}